\documentclass[11pt]{article}
\usepackage{enumerate}

\usepackage{amsmath}
\usepackage{amsfonts}
\usepackage{amssymb}
\usepackage[english]{babel}
\usepackage{graphicx,epsf,cite,esint}
\usepackage{amsthm}
\usepackage{mathtools}
\usepackage{textcomp}

\usepackage{hyperref}
\usepackage{accents}
\usepackage{calrsfs}
\usepackage[T1]{fontenc}
\usepackage[utf8]{inputenc}
\usepackage{empheq}

\def\XXint#1#2#3{{\setbox0=\hbox{$#1{#2#3}{\int}$ }
\vcenter{\hbox{$#2#3$ }}\kern-.6\wd0}}

\usepackage [autostyle, english = american]{csquotes}
\MakeOuterQuote{"}
\numberwithin{equation}{section}
\theoremstyle{plain}
\newtheorem*{theorem*}{Theorem}
\newtheorem*{lemma*}{Lemma}
\newtheorem{theorem}{Theorem}
\newtheorem{lemma}{Lemma}[section]

\newtheorem{proposition}[lemma]{Proposition}

\theoremstyle{definition}
\newtheorem{definition}[lemma]{Definition}
\newtheorem{remark}[lemma]{Remark}

\def\RR{\mathbb{R}}
\newcommand{\pd}{\partial}

\newcommand{\MM}{\mathcal{M}_a}

\newcommand{\beq}{\begin{equation}}
\newcommand{\eeq}{\end{equation}}

\newcommand{\bal}{\begin{align}}
\newcommand{\eal}{\end{align}}
\newcommand{\bals}{\begin{align*}}
\newcommand{\eals}{\end{align*}}

\usepackage{geometry}
\begin{document}
\title{Stable self-similar blowup for a family of nonlocal transport equations}
\author{ Tarek M. Elgindi, Tej-eddine Ghoul, Nader Masmoudi}
\maketitle

\begin{abstract}
We consider a family of non-local problems that model the effects of transport and vortex stretching in the incompressible Euler equations. Using modulation techniques, we establish \emph{stable} self-similar blow-up near a family of known self-similar blow-up solutions. 
\end{abstract}
\section{Introdution}
The dynamics of solutions to the 3D incompressible Euler equations is guided by many effects which are still not properly understood. Among these effects are: \begin{itemize}
\item Non-locality,
\item Transport,
\item Vortex Stretching.
\end{itemize}
"Non-locality" is physically clear: in an ideal fluid, any disturbance in one location is immediately felt everywhere. "Transport" refers to the fact that while vortices produce a velocity field, they are also carried by that velocity field to different locations in space. "Vortex Stretching" is the process by which vortices are enhanced due to variations in the velocity gradient in the direction of the vortex. This is succinctly captured in the 3D Euler system as follows: 
\begin{equation}
\left\{\begin{array}{l}
\pd_t\omega+u\cdot\nabla\omega=\omega\cdot\nabla u,\\
u=\nabla\times(-\Delta)^{-1}\omega.
\end{array}
 \right.
\end{equation} 
Non-locality is described by the relation $u=\nabla\times(-\Delta)^{-1}\omega$ (called the Biot-Savart law) while the time-evolution of the vorticity $\omega$ is determined by the transport term $u\cdot\nabla\omega$ and the vortex-stretching term $\omega\cdot\nabla u$. Notice that the incompressible Euler equation is a system of three equations and that each equation contains seven terms all coupled together through the non-local Biot-Savart law. Many authors have written about the different effects of each term, observed through numerical simulations \cite{HouLei}, the construction of special solutions \cite{Stuart1988,EJSI}, and the analysis of model problems \cite{CLM,DeGregorio1996,Con,OSW}. Since a finite-time singularity in the Euler equation can only happen if the magnitude of the vorticity $\omega$ becomes unbounded, many have highlighted the vortex stretching as \emph{the} source of a possible singularity. 

Wanting to understand better the qualitative nature of the vortex stretching term, Constantin, Lax and Majda \cite{CLM} introduced the following model set on $\mathbb{R}\times [0,T_*)$:
\begin{equation}\label{eq:CLMeq}
\left\{\begin{array}{l}
\pd_t\omega =-2H(\omega)\omega\\
H(\omega)(x)=\frac{1}{\pi}PV\int_{-\infty}^\infty \frac{\omega(y)}{x-y}dy.
\end{array}
 \right.
\end{equation} From the 3D Euler equation, one just drops the transport term, makes the system into a single equation, and approximates $\nabla u=\nabla\, \nabla\times (-\Delta)^{-1}\omega$ by a zeroth order operator. In one dimension, the natural choice is the Hilbert transform $H$.
With these simplifications, they solve explicitly the equation and prove that solutions can become singular in finite time. This allows one to speculate that singularity formation is possible even in the 3D Euler equation. There is at least one major problem, however, with this model: while the transport term does not change the magnitude of the vorticity, it can counteract the growing effects of the vortex stretching term. This can be seen easily in the following simple model: \[\partial_t\omega+\lambda(t) \sin(x)\partial_x\omega=\lambda(t)\omega,\] set on $[-\pi,\pi]$ with periodic boundary conditions and with $\lambda(t)$ time-dependent constant. It is not difficult to show that if $\omega_0$ is $C^1$ and vanishes at $0$ and $\pi$, then solutions to this equation are \emph{uniformly} bounded in time independent of the size of $\lambda(t)$. In particular, transport can act to "deplete" the growth effects of the vortex stretching term in this simple model. Thus, while the transport term cannot cause a singularity, it can stop it from happening. 

After the work \cite{CLM},  De Gregorio \cite{DeGregorio1996} introduced a model that takes into account both effect vortex stretching and transport,
\begin{equation}\label{eq:DeGregorio}
\left\{\begin{array}{l}
\pd_t\omega+2u\partial_x \omega=2\partial_x u \omega\\
u=-\Lambda^{-1}\omega(x)=-\int_0^xH\omega(y)dy.
\end{array}
 \right.
\end{equation}
Based on numerical simulations, De Gregorio conjectured by that the addition of the transport term should lead to global regularity. Strong evidence for this conjecture has been given in \cite{JSS} and global regularity for special kind of data was given in \cite{LeiLiuRen}. Inspired by this conjecture,  Okamoto, Sakajo, and Wunsch \cite{OSW} introduced a new model where they weight the transport term with a  parameter $a$.
\begin{equation}\label{eq:OSW}
\left\{\begin{array}{l}
\pd_t\omega+a\,u\partial_x \omega=2\partial_x u \omega\\
u=-\Lambda^{-1}\omega(x)=-\int_0^xH\omega(y)dy.
\end{array}
 \right.
\end{equation}The purpose of this model was to understand the effects of the modeled vortex stretching and transport terms. 
Hence, when $a=2$ we get the De Gregorio model and when $a=0$ we get CLM model.
In the same idea of \cite{CLM}, Cordoba, Cordoba and Fontelos \cite{CCF} introduced a 1D model to mimic the 2D quasi-geostrophic equation:
\begin{equation}\label{eq:CCF}
\left\{\begin{array}{l}
\pd_t\omega+2H(\omega)\omega (x,t) =-2\Lambda^{-1}\omega\pd_x\omega\in\RR\times[0,T_*)\\
H(\omega)(x)=\frac{1}{\pi}PV\int_{-\infty}^\infty \frac{\omega(y)}{x-y}dy,
\end{array}
 \right.
\end{equation}
which corresponds to $a=-2$ in the generalized model \eqref{eq:OSW}.

Recently, Elgindi and Jeong \cite{EJDG} proved the existence of a smooth self-similar profile for $a$ small by using a local continuation argument.
The goal here is to prove the stability of those profile for all $a$ small enough. The proof is based on the modualtion technique which has been developed by Merle, Raphael, Martel, Zaag and others. This technique has been very efficient to describe the formation of singularities for many problems like the nonlinear wave equation \cite{MR3302641}, the nonlinear heat equation \cite{MR1427848}, reaction diffusion systems \cite{MR3779644,MR3846237}, the nonlinear Schrodinger equation \cite{MR2150386,MR2257393}, the GKDV equation \cite{MR3179608}, and many others.
Note that for \eqref{eq:OSW} comparing to all the previous models cited above there exists a group of scaling transformations of dimension larger than two that leaves the equation invariant. Here this degeneracy is a real difficulty since one does not know in advance which scaling law the flow will select. We remark that similar results to Theorems \ref{theorem:main} and \ref{theorem:main2} were recently established in a work of Chen, Hou, and Huang \cite{CHH}. The authors of \cite{CHH} are also able to find a singularity for the De Gregorio ($a=2$) model on the whole line using a very interesting argument with computer assistance.  

\subsection*{Main Theorem}
We introduce first the following weighted space
$$L^2_\phi(\RR)=\{f\in L^2_{loc}(\RR):\int_\RR |f|^2\phi<\infty\},$$
equipped with the norm and inner product
$$\|f\|^2_{L^2_\phi}=(f\phi,f)_{L^2},$$
where $\phi=\frac{(1+y^2)^2}{y^4}$.
From now on we focus on \eqref{eq:OSW}.
In \cite{EJDG}, the authors show the existence of self-similar solutions of \eqref{eq:OSW} of the form
\begin{align}
\omega(x,t)=\frac{1}{T-t}F_a\Big(\frac{x}{(T-t)^{1+\gamma(a)}}\Big),
\end{align}
where $F_a$ solves 
\begin{align}\label{Faeq}
F_a+((1+\gamma(a))y-a\Lambda^{-1}F_a)F_a'+2HF_aF_a=0,
\end{align}
 and $\gamma(a)=a(-2+\ln(4))+O(a^2)$.
When $a=0$, the profile $F_0$ has the form: \[F_0(y)=\frac{y}{1+y^2}, \qquad HF_0(y)=-\frac{1}{y^2+1},\] while for $|a|$ small \[|\gamma(a)|+|F_a-F_0|_{H^3}\leq C |a|\] with $C>0$ a universal constant. In fact, we also have the following expanions:
\begin{align}\label{defF_a}
F_a(y)=\left\{\begin{array}{l}
y +O(y^3),  \quad y\leq1,\\
C_1|y|^{-\frac{1}{1+\gamma(a)}}+O(|y|^{-\frac{2}{1+\gamma(a)}}),  \quad y\geq1,
\end{array}
\right.
\end{align} 
\begin{align}\label{defHF_a}
HF_a(y)=\left\{\begin{array}{l}
-\frac{2+\gamma(a)}{2-a} +O(y^2),  \quad y\leq1,\\
C_2|y|^{-\frac{1}{1+\gamma(a)}}+O(|y|^{-\frac{2}{1+\gamma(a)}}),  \quad y\geq1.
\end{array}
\right.
\end{align}
The main result of this work is the dynamic stability of these blow-up profiles. In particular, this allows us to construct compactly supported solutions with local self-similar blow-up and cusp formation in finite time (a phenomenon numerically conjectured to occur in the case $a=-2$).

To prove the stability of the profiles $F_a$ we rescale \eqref{eq:OSW}.
A natural change of variables to do here will be 
\begin{align}\label{changeofvariable}
z&=\frac{x}{\lambda^{1+\gamma(a)}},\quad\quad\frac{ds}{dt}=\frac{1}{\lambda},\nonumber\\
 \omega(x,t)&=\frac{1}{\lambda}v\Big(\frac{x}{\lambda^{1+\gamma(a)}},s\Big).
\end{align}
Hence, in these new variables we get the following equation on $v$:
\begin{align}\label{CCF_ys}
\left\{\begin{array}{l}
v_s-\frac{\lambda_s}{\lambda}(v+(1+\gamma(a))zv_z) +2H(v)v=a\Lambda^{-1}vv_z\in\RR\times[0,\infty)\\
H(v)(y)=\frac{1}{\pi}PV\int_{-\infty}^\infty \frac{v(x)}{z-x}dx.
\end{array}
 \right.
\end{align}
Note that this change of variable leaves the $C^{\alpha(a)}$ norm of the velocity $u=-\Lambda^{-1}\omega$ unchanged, with $\alpha(a)=1-\frac{1}{1+\gamma(a)}$.
This indicates that the velocity $u$ will form a $C^\alpha$ cusp.
Note that \eqref{Faeq} is invariant under the following scaling:
$$ F_{a,\mu}(z):=F_a(\mu z).$$
We will make an abuse of notation by denoting $F_{a,\mu}$ by $F_a$.
Actually, this will induce an instability as one can see on the spectrum of the linearized operator around the profile $F_a$ in \eqref{eigenL}.
To fix this instability, we will allow $\mu$ to depend on time and fix it through an orthogonality condition.
Hence, we introduce
$$v(z)=w(\mu z),~~y=\mu z$$
where $w$ solves
\begin{align}\label{eqw}
w_s+\frac{\mu_s}{\mu}yw_y-\frac{\lambda_s}{\lambda}(w+(1+\gamma(a))yw_y) +2H(w)w=a\Lambda^{-1}ww_y,
\end{align}
and
$$\mathcal{S}_a(w)=w+(1+\gamma(a))y\pd_y w.$$
Now we linearize around $F_a$ by setting,
$$w=F_a+q,$$
where $q$ solves
\begin{align}\label{lineareq}
q_s+\frac{\mu_s}{\mu}y(F_a'+q_y)-\Big(\frac{\lambda_s}{\lambda}+1\Big)\mathcal{S}_a(F_a+q)=\MM q-2Hqq+a\Lambda^{-1}q q_y,
\end{align}
where
\begin{align}\label{defopL}
\MM q=-(2HF_a+1)q-2Hq F_a-((1+\gamma(a))y-a\Lambda^{-1}F_a)q_y+a\Lambda^{-1}qF_a'.
\end{align}
\begin{theorem}\label{theorem:main}
Let $a$ be small enough and $\tilde{\alpha}(a)=1-\frac{1}{1+\gamma(a)}$, then there exists an open set of odd initial data of the form $\omega(t=0)=F_a+q_0$ with $\pd_x q_0(0)=Hq_0 (0)=0$ and where 
\begin{align}
C_2\int_\RR|q_0(y)|^2\phi(y)dy+C_1\int_\RR|y\pd_y q_0|^2\phi(y)dy<\epsilon,
\end{align}
such that there exists $C(q_0)>0$ and $\omega$ verifies
$$\omega(x,t)=\frac{1}{\lambda(t)}\Bigg(F_a\Big(\frac{x\mu(t)}{\lambda^{1+\gamma(a)}}\Big)+q\Big(\frac{x\mu(t)}{\lambda^{1+\gamma(a)}},t\Big)\Bigg),$$
with $\frac{\lambda(t)}{T-t}\rightarrow C(q_0)$ and $\mu(t)\rightarrow\mu^*$ as $t\rightarrow T$ and
$$C_1\int_\RR|y\pd_y q(y,t)|^2 \phi(y)dy+C_2\int_\RR|q(y,t)|^2 \phi(y)dy\lesssim_\delta (T-t)^{1-\delta},$$  for any $\delta>0$.
When $a<0$ (so that $\gamma(a)>0$) we have $$\sup_{t\in[0,T]}\|u(t)\|_{C^{\tilde{\alpha}(a)}}<\infty, \mbox{ with } u(t,x)=-\Lambda^{-1}\omega.$$
\end{theorem}

\begin{remark}
The last statement of the theorem on the $C^\alpha$ regularity of the velocity field when $a<0$ is related to the $C^{1/2}$ conjecture made in \cite{Kis} and \cite{SV}.
\end{remark}

\begin{remark}
The assumptions of the theorem do not require that $q_0$ be differentiable everywhere (just that $q\in H^1$ and that $q$ vanishes to high order near $0$). Note also that the assumptions that $\partial_x q_0(0)=Hq_0(0)=0$ can be trivially removed using Lemma \ref{lemma:mod} since $F_a$ could be replaced by a slightly rescaled version of $F_a$ to make the perturbation and its Hilbert transform vanish to second order at 0. 
\end{remark}

\begin{remark}
Note that the open set of initial data contains slowly decaying solution, but also compactly supported solutions.
Indeed, one can impose that $q_0\sim-F_a$ at infinity.
\end{remark}
Our next theorem is in the same spirit as Theorem \ref{theorem:main}. The difference is that it applies to the non-smooth self-similar solutions constructed in \cite{EJDG} and, thus, applies even for $a$ large. Indeed, in \cite{EJDG}, the authors constructed a family of self-similar solutions to \eqref{eq:OSW} which are smooth functions of the variable $X=|x|^\alpha$ with speed $1+\tilde\gamma_\alpha(a)$ whenever $|a\alpha|$ is small enough. Denoting these solutions by $F_a^\alpha$, we have the following stability theorem. 
\begin{theorem}\label{theorem:main2}
Define \[\phi_*(Y)=\frac{(1+Y)^4}{Y^4}.\]  There exists $c_0>0$ so that if $a\in\RR$ and $\alpha>0$ satisfy $(|a|+1)\alpha<c_0$ and $\tilde{\beta}_\alpha(a)=1-\frac{1}{1+\tilde\gamma_\alpha(a)}$, then there exists an open set of odd initial data of the form $\omega(t=0)=F_a^\alpha+q_0$ with $\pd_X q_0(0)=Hq_0 (0)=0$ and where 
\begin{align}
C_2\int_\RR|q_0(Y)|^2\phi_*(Y)dY+C_1\int_\RR|Y\pd_Y q_0|^2\phi_*(Y)dY<\epsilon,
\end{align}
such that there exists $C(q_0)>0$ and $\omega$ verifies
$$\omega(x,t)=\frac{1}{\lambda(t)}\Bigg(F_a\Big(\frac{X\mu(t)}{\lambda^{1+\tilde\gamma_\alpha(a)}}\Big)+q\Big(\frac{X\mu(t)}{\lambda^{1+\tilde\gamma_\alpha(a)}},t\Big)\Bigg),$$
with $\frac{\lambda(t)}{T-t}\rightarrow C(q_0)$ and $\mu(t)\rightarrow\mu^*$ as $t\rightarrow T$ and
$$C_1\int_\RR|Y\pd_Y q(Y,t)|^2 \phi_*(Y)dy+C_2\int_\RR|q(Y,t)|^2 \phi_*(Y)dY\lesssim_\delta (T-t)^{1-\delta},$$ for any $\delta>0$ and 
$$\sup_{t\in[0,T]}\|u(t)\|_{C^{\tilde{\beta}_\alpha(a)}}<\infty, \mbox{ with } u(t,x)=-\Lambda^{-1}\omega.$$
\end{theorem}
\begin{remark}
Note that in Theorem \ref{theorem:main2} we allow the parameter $a$ to be anything in $\RR$ but we pay the price on the regularity, since we need to pick $\alpha$ so that $|a\alpha|$ is small enough. Note also that $\tilde\beta_\alpha(a)\rightarrow 1$ as $\alpha\rightarrow 0$ for any fixed $a$. This means that as $\alpha\rightarrow 0$, the blow-up becomes more and more mild. 
\end{remark}
\begin{remark}
The proof of Theorem \ref{theorem:main2} is sketched in Section \ref{Proof of Theorem 2}; the only main difference between the proofs of Theorems \ref{theorem:main} and \ref{theorem:main2} is the coercivity of the linearized operator and an extra change of variables in the proof of Theorem \ref{theorem:main2}. 
\end{remark}

\subsection*{Organization of the paper}
In the Section \ref{Coercivity} we establish coercivity estimates for the linearized operator $\mathcal{M}_a$  under the assumption that the perturbation $q$ vanishes to high order at $0$ along with its Hilbert transform. This is the core of the argument. In the following Section \ref{Modulation}, we modulate the free parameters $\lambda$ and $\mu$ to propagate the vanishing condition on $q$. Then we prove long-time decay estimates on $q$ (in self-similar variables) in Section \ref{Energy} which show that the perturbation $q$ becomes small relative to the self similar profile as we approach the blow-up time. We establish Theorem \ref{theorem:main2} in Section \ref{Proof of Theorem 2}.
\section{Coercivity}\label{Coercivity}

\begin{proposition}\label{proposition:spectralgap}
There exists a universal constant $C>0$ so that if $a$ is small enough and if $f$ is odd, $f'(0)=Hf(0)=0$ and
$$\int_\RR|f|^2\phi(y)dy<+\infty,$$
\begin{align}\label{spectralgap}
\int_\RR f\MM f \phi(y)dy&\leq-\Big(\frac{1}{2}-C\,|a|\Big)\int_\RR f(y)^2 \phi(y)dy.
\end{align}
\end{proposition}

The proof of this lemma requires a weighted identity for the Hilbert transform which we show in Lemma \ref{WeightedIdentity}.

\begin{proof}
We write: \[\MM f=-(2HF_a+1)f-2Hf F_a-((1+\gamma(a))y-a\Lambda^{-1}F_a)f_y+a\Lambda^{-1}fF_a'\]
\[=-(2HF_0+1)f-2Hf F_0-y f_y+a\overline\MM f.\]
Observe that if $a$ is small enough, there exists a universal constant $C>0$ (independent of $a$) so that we have the following estimate: \[\Big|\int_\mathbb{R}f\overline\MM f \phi(y)dy\Big|\leq C\int (|Hf|^2+|f|^2) \phi(y)dy\leq C\int |f|^2\phi(y)dy,\] using Lemma \ref{WeightedIdentity}. 
This follows from the following observation: \[|\frac{1}{\phi(y)}\partial_y(\phi(y)\Lambda^{-1} F_a)|_{L^\infty}+|{F_a}'|_{L^\infty}+\frac{|\gamma(a)|}{|a|}+\frac{1}{|a|}|F_a-F_0|_{L^\infty}+\frac{1}{|a|}|HF_a-HF_0|_{L^\infty}\leq C.\]
The only one which is not a direct consequence of the expansion given in \cite{EJDG} is the first one which we see can be estimated by:
\[|\frac{1}{\phi(y)}\partial_y \phi(y)\Lambda^{-1}F_a|_{L^\infty}+|HF_a|_{L^\infty}\leq |\frac{1}{y}\Lambda^{-1} F_a|_{L^\infty}+|HF_a|_{L^\infty}\leq C|HF_a|_{L^\infty}\leq C.\]
Thus, we must consider only the quantity: \[\int \Big((-2HF_0-1)f-2HfF_0-yf_y\Big)f \phi dy.\] First let us observe: \[\int Hf f F_0 \phi=0.\]
Indeed, \[\int Hf f F_0 \phi=\int fHf  \frac{y^2+1}{y^3}=\int \frac{fHf}{y}+\int \frac{fHf}{y^3}=H(fHf)(0)+\frac{1}{2}H(\partial_{yy} (fHf))(0)\] \[=\frac{1}{2} (H(f)^2(0)-f(0)^2)+\frac{1}{4}\partial_{yy}(H(f)^2-f^2)(0)=0,\] by the assumptions\footnote{Note that, strictly speaking, $fH(f)$ is not twice differentiable but the equality $\int Hf f F_0\phi=0$ can be made rigorous by applying the smoothing procedure in Lemma \ref{Smoothing}.} on $f$. This leaves us with:  \[\int (-2HF_0-1)f^2 \phi(y)-yf f_y \phi dy=\int \big(-2HF_0-1+\frac{1}{2}\frac{\partial_y (y\phi)}{w}\big) |f(y)|^2\phi(y)dy.\]
Next observe that \[-2HF_0-1+\frac{1}{2}\frac{\partial_y \phi}{\phi}=\frac{2}{1+y^2}-1+\frac{y^2-3}{2(y^2+1)}=-\frac{1}{2}.\]
This completes the proof. 
\end{proof}

\section{Modulation equation and derivation of the law}\label{Modulation}
Since our coercivity estimate from the previous section relies on $\partial_y q(0)=H(q)(0)=0$, we will use that we have the "free" parameters $\mu$ and $\lambda$ to fix these conditions. To find precisely how to do this, we will just differentiate \eqref{lineareq} with respect to $y$ and apply the Hilbert transform to \eqref{lineareq} and evaluate both at $y=0$. 
We will prove now by using the implicit function theorem that there exists a unique decomposition to the solution $w$ of \eqref{eqw}.
Indeed, in the following Lemma we fix $\mu$ and $\lambda$ such that $q_{y}(s,0)=Hq(0)=0$.
\begin{lemma}\label{lemma:mod}[Modulation]
For $q\in L^1_{loc}$ for which $Hq(0)$ and $q'(0)$ exist and \[|Hq(0)|+|q'(0)|\leq \frac{1}{2},\] there exists a unique pair $(\mu, \lambda)\in (0,\infty)^2$ so that 
$$\tilde q:=F_a(y)+q-\tilde{F_a}_{\mu,\lambda}$$
with $$\tilde{F}_{a,\mu,\lambda}(y)=\frac{1}{\lambda}F_a\Big(\frac{y\mu}{\lambda^{1+\gamma(a)}}\Big),$$
satisfies
$$\tilde{q}_y(0)=H\tilde{q}(0)=0.$$
 In fact, \[\lambda=1-H(q)(0)\frac{(2+\gamma(a))}{2-a}\qquad \text{and} \qquad \mu=(1+q'(0))\lambda^{2+\gamma(a)}.\] 
\end{lemma}

\begin{proof}[Proof of Lemma \ref{lemma:mod}]
We want to find $\mu,\lambda$ so that \[\tilde q:=F_a+q-\tilde F_{a\mu,\lambda}\] satisfies $\tilde q_y(0)=H\tilde q(0)=0$. 
Observe that \[ (\tilde F_{a\mu,\lambda})'(0)=\frac{\mu}{\lambda^{2+\gamma(a)}} F_a'(0)\qquad H(\tilde F_{a\mu,\lambda})(0)=\frac{1}{\lambda}HF_a(0)\] while \[F_a'(0)=1\qquad H F_a(0)=-\frac{2+\gamma(a)}{2-a}.\]

\end{proof}
Let $w_0,y\pd_y w_0$ be in $L^2_\phi(\RR)$ with a small enough norm and let $w$ be its corresponding solution.
Consequently, thanks to Lemma \ref{lemma:mod} the solution admits a unique decomposition on some time interval $[s_0,s^*):$ 
\begin{align}
w(y,s)=\tilde{F}_{a,\mu,\lambda}+q,
\end{align}
where $$q_{y}(s,0)=Hq(s,0)=0.$$

\subsection{The bootstrap regime} 

We will define first in which sense the solution is initial close to the self-similar profile. 

\begin{definition}[Initial closeness]\label{definition:ini} 
Let $\delta>0$ small enough, $s_0\gg 1$, and $w_0\in H^1_\phi$. We say that $w_0$ is initially close to the blow-up profile if there exists $\lambda_0>0$ and $\mu_0>0$ such that the following properties are verified. In the variables $(y,s)$ one has:
\begin{align}
w_0(y)=F+q_0,
\end{align}
and the remainder and the parameters satisfy:
\begin{itemize}
\item[(i)] \emph{Initial values of the modulation parameters:}
\begin{align} \label{parametersini}
\frac 12 e^{\frac{s_0}{2}}< \lambda_0 < 2 e^{\frac{s_0}{2}}, \ \ \frac 12 <\mu_0 < 2
\end{align}
\item[(ii)] \emph{Initial smallness:}
\begin{align} \label{qini}
\|y\pd_y q_0\|_{L^2_\phi}^2+\frac{1}{\delta}\|q_0\|_{L^2_\phi}^2< e^{-\frac{s_0}{8}}, 
\end{align}
\end{itemize}
\end{definition}

We are going to prove that solutions initially close to the self-similar profile in the sense of Definition \ref{definition:ini} will stay close to this self-similar profile in the following sense.

\begin{definition}[Trapped solutions] \label{definition:trap}

Let $K\gg 1$. We say that a solution $w$ is trapped on $[s_0,s^*]$ if it satisfies the properties of Definition \ref{definition:ini} at time $s_0$, and if it can be decomposed as  
$$w=F+q(y,s),$$
 for all $s\in [s_0,s^*]$ with:
\begin{itemize}
\item[(i)] \emph{Values of the modulation parameters:}
\begin{align} \label{parameterstrap}
\frac 1K e^{-\frac{s}{8}}< \lambda(s) < K e^{-\frac{s}{8}}, \ \ \frac 1K <\mu(s) < K.
\end{align}
\item[(ii)] \emph{Smallness of the remainder:}
\begin{align} \label{etrap}
\|y\pd_y q\|_{L^2_\phi}^2+\frac{1}{\delta}\|q\|_{L^2_\phi}^2< Ke^{-\frac s8},
\end{align}
\end{itemize}
\end{definition}

\begin{proposition} \label{proposition:bootstrap}
There exist universal constants $K,s_0^*\gg 1$ such that the following holds for any $s_0\geq s_0^*$. All solutions $w$ initially close to the self-similar profile in the sense of Definition \ref{definition:ini} are trapped on $[s_0,+\infty)$ in the sense of Definition \ref{definition:trap}.
\end{proposition}
Define for $\delta>0$ small enough:
\begin{align}\label{defenergy}
\mathcal{E}(s)=\|y\pd_y q\|_{L^2_\phi}^2+\frac{1}{\delta}\|q\|_{L^2_\phi}^2.
\end{align}

The proof of the proposition will be done later by using energy estimates. Before this we will derive that "law" that $\mu$ and $\lambda$ will satisfy. 

Indeed, we will prove the following
\begin{proposition}\label{proposition:derivationofthelaw}
To ensure that \[q_{y}(s,0)=Hq(0)=0,\] it suffices to impose that $\mu$ and $\lambda$ satisfy the following two relations:
\begin{align}\label{modbound}
\frac{\mu_s}{\mu}&=(2+\gamma(a))\Big(\frac{\lambda_s}{\lambda}+1\Big),\\
\Big(\frac{\lambda_s}{\lambda}+1\Big)\frac{2+\gamma(a)}{2-a}&=a\Big( H(\Lambda^{-1}F_a q_y)(0,s)+H(\Lambda^{-1}qF_a')(0,s)+H(\Lambda^{-1}qq_y)(0,s)\Big).\\
\Big|\frac{\lambda_s}{\lambda}+1\Big|&\leq C|a|\sqrt{\mathcal{E}}.
\end{align}
\end{proposition}
\begin{proof}

Dividing \eqref{lineareq} by $y$ and evaluating at $y=0$ and using that $q$ is odd we get: \[\partial_s (q_y(0,s))+\frac{\mu_s}{\mu}(F_a'(0)+q_y(0))-
\Big(\frac{\lambda_s}{\lambda}+1\Big)\partial_y\Big((1+(1+\gamma(a))y\partial_y)(F_a +q)\Big)\Big|_{y=0}\]\[=\partial_y\mathcal{M}_a q\Big|_{y=0}-2\partial_y(Hq q)\Big|_{y=0}+a\partial_y(\Lambda^{-1}q q_y)\Big|_{y=0}.\]
By inspection, using that $F_a'(0)=1$we see that \[\partial_s(q_y(0,s))+\frac{\mu_s}{\mu}-(2+\gamma(a))\Big(\frac{\lambda_s}{\lambda}+1\Big)=AHq(0,s)+Bq_y(0,s),\] for some $s$-dependent numbers $A, B$ depending only on $q$ ,$F_a,$ $\mu$ and $\lambda$. 
Before finding the second law, we note the following simple fact for decaying functions $f$ with $f_y\in L^1$: \[H(y\partial_y f)(0)=\int f_y=0.\] Now we apply $H$ to \eqref{lineareq} and evaluating at $y=0$ we will get the second law: \[\partial_s (Hq(0,s))-\Big(\frac{\lambda_s}{\lambda}+1\Big)(HF_a(0)+Hq(0,s))\]\[=-H(q)(0)+2H(q)(0,s)H(F_a)(0)+aH(\Lambda^{-1}F_a q_y)(0,s)+aH(\Lambda^{-1}qF_a')(0,s)-Hq(0,s)^2+aH\Big(\Lambda^{-1}(q)q_y\Big)(0,s)\]
In particular, if \[\Big(\frac{\lambda_s}{\lambda}+1\Big)\frac{2+\gamma(a)}{2-a}=a\Big( H(\Lambda^{-1}F_a q_y)(0,s)+H(\Lambda^{-1}qF_a')(0,s)\Big)+H(\Lambda^{-1}qq_y)(0,s),\] \[\frac{\mu_s}{\mu}=(2+\gamma(a))\Big(\frac{\lambda_s}{\lambda}+1\Big),
\]
and if $Hq(0,0)=q_y(0,0)=0$ we get \[Hq(0,s)=q_y(0,s)=0,\] for all $s\geq 0$. 
In addition, one get that
\begin{align}
\Bigg|\frac{\lambda_s}{\lambda}+1\Bigg|\lesssim \int_0^\infty\frac{|a|(|\Lambda^{-1}F_aq_y|+||\Lambda^{-1}q F'_a|)+|\Lambda^{-1}qq_y|}{y}dy.
\end{align}
By using that \[\Big|\frac{\Lambda^{-1}q}{y}\Big|_{L^\infty}\lesssim |H(q)|_{L^\infty}\lesssim |q|_{H^1}\lesssim \sqrt{\mathcal{E}(s)} \].
We deduce that,

$$\Bigg|\frac{\lambda_s}{\lambda}+1\Bigg|\leq |a|C\sqrt{\mathcal{E}(s)}.$$
\end{proof}

\section{Energy Estimates}\label{Energy}

The goal of this section is to establish energy estimates for $q$ in a suitable space. Let us first define our energy: \[\mathcal{E}(q)=\frac{1}{\delta}\int |q|^2 \phi(y)dy+\int |y\partial_y q|^2\phi(y)dy,\] where $\delta$ will be chosen to be small enough. 
We will prove that if $a$ is small enough, then 
\begin{align}\label{Energyineq}
\frac{d}{ds}\mathcal{E}(q)\leq -\frac{1}{4}\mathcal{E}(q)+C\mathcal{E}(q)^{3/2},
\end{align}

 for a universal constant $C>0$. 

Observe that from our choice of $\mu$ and $\lambda$ in Proposition \ref{proposition:derivationofthelaw}, we have the following estimate (again assuming that $a$ is small enough):
\[\Big|\frac{\lambda_s}{\lambda}+1\Big|\leq a\,C(\sqrt{\mathcal{E}(q)}+\mathcal{E}(q))\] for some universal constant $C>0$. Next, we use $\frac{\mu_s}{\mu}=(2+\gamma(a))\Big(\frac{\lambda_s}{\lambda}+1\Big)$ in \eqref{lineareq} to deduce
\begin{align}\label{MIeq}
&q_s+\Big(\frac{\lambda_s}{\lambda}+1\Big)(F_a-yF_a'+q-yq_y)=\MM q-2Hqq+a\Lambda^{-1}qq_y,
\end{align}
Taking the (weighted) inner product of \eqref{MIeq} with $q$ we get: \[\frac{1}{2}\frac{d}{ds}\|q\|_{L^2_\phi}^2\leq (\mathcal{M}_a q,q)+aC(\mathcal{E}(q)+\mathcal{E}(q)^{\frac{3}{2}})-2(H(q)q,q)_{L^2_\phi}+a(\Lambda^{-1} q q_y, q)_{L^2_\phi},\] where we have used that \[\Big|\frac{\partial_y(y\phi(y))}{\phi(y)}\Big|_{L^\infty}+|F_a-yF_a'|_{L^2_\phi}\leq C,\] for some universal constant $C$ independent of $a$.  Now, we have \[|H(q)|_{L^\infty}\leq C|q|_{H^1}\leq C\mathcal{E}(q).\] Furthermore, 
\[(\Lambda^{-1}q q_y, q \phi)=-\frac{1}{2}(q^2\partial_y(\phi\Lambda^{-1} q))=-\frac{1}{2}(q^2\Lambda^{-1}q,\partial_y \phi)-\frac{1}{2}(q^2,H(q)\phi).\]
Now observe that \[|\frac{\Lambda^{-1}q}{y}|_{L^\infty}\leq C|H(q)|_{L^\infty}\] and \[\Big|\frac{y\partial_y \phi}{\phi}\Big|\leq C.\] Thus, \begin{equation}\frac{1}{2}\frac{d}{ds}\|q\|_{L^2_\phi}^2\leq -\Big(\frac{1}{2}-C|a|\Big)|q|_{L^2_\phi}^2+|a|C\mathcal{E}(q)+C\mathcal{E}(q)^{\frac{3}{2}}.\end{equation}
Now we establish the first derivative estimate. First we apply $y\partial_y$ to \eqref{MIeq} and get: 

\[\partial_s(yq_y)+\Big(\frac{\lambda_s}{\lambda}+1\Big)(-y^2F_a''-y^2q_{yy})= \MM (y q_y)+[\MM,y\partial_y]q-2y\partial_y(Hqq)+ay\partial_y(\Lambda^{-1}qq_y).\]
Now we multiply by $y \phi q_y$ and integrate to get: \[\frac{1}{2}\frac{d}{ds}\|y\partial_yq\|^2_{L^2_\phi}\leq (\mathcal{M}_a (yq_y),yq_y)_{L^2_\phi}+([\mathcal{M}_a, y\partial_y]q, yq_y)_{L^2_\phi}-2(y\partial_y(Hqq),yq_y)_{L^2_\phi}+(ay\partial_y(\Lambda^{-1}q q_y),yq_y)_{L^2_\phi}\] \[+|a|C(\mathcal{E}(q)+\mathcal{E}(q)^{\frac{3}{2}}).\] Note that $yq_y$ is an odd function and $H(yq_y)(0)=0=(yq_y)_y(0).$ Thus, \[(\mathcal{M}_a (yq_y),yq_y)_{L^2_\phi}\leq -\Big(\frac{1}{2}-aC\Big)\|yq_y\|_{L^2_\phi}^2,\] using Proposition \ref{spectralgap}. 
It is also easy to see that, as before,  \[|(y\partial_y(Hqq),yq_y)_{L^2_\phi}|+|(y\partial_y(\Lambda^{-1}q q_y),yq_y)_{L^2_\phi}|\leq C\mathcal{E}(q)^{3/2}\] once we observe that $yH(\partial_y q)=H(y\partial_y q)$ and remember that $H$ is an isometry on odd functions in $L^2_\phi$ whose first derivative and Hilbert transform vanish at 0. Thus it remains to estimate the commutator term: \[([\mathcal{M}_a, y\partial_y]q, yq_y)_{L^2_\phi}.\] Note that $y$ and $\partial_y$ both commute with the Hilbert transform (when the argument of the Hilbert transform is an odd function). Let us first recall the form of $\MM:$  \[\MM q=-(2HF_a+1)q-2Hq F_a-((1+\gamma(a))y-a\Lambda^{-1}F_a)q_y+a\Lambda^{-1}qF_a'.\]
In particular, \[[\MM,y\partial_y]q=\sum_{i=1}^6 I_i,\] where \[I_1=-2q\, y\partial_y HF_a,\qquad I_2=-2Hq y\partial_y F_a, \qquad I_3=a \, q_y \,y\partial_y \Lambda^{-1} F_a,\]
\[I_4=-a\Lambda^{-1}F_a \partial_y q, \qquad I_5=a\Lambda^{-1} q yF_a'',\qquad I_6=aF_a'[\Lambda^{-1},y]\partial_yq.\]
Thus we see readily: \[|(I_1+I_2, yq_y)_{L^2_2}|\leq C|q|_{L^2_\phi}|yq_y|_{L^2_\phi}\Big(|y\partial_y F_a|_{L^\infty}+|y\partial_y HF_a|_{L^\infty}\Big)\leq C|q|_{L^2_\phi}|yq_y|_{L^2_\phi}.\]
Moreover, \[|(I_3+I_4+I_5, yq_y)_{L^2_\phi}|\leq |a| \mathcal{E}(q).\]
Now note that 
\begin{align}
[\Lambda^{-1},y]\partial_y q&=y\int_0^y \partial_z H(q)(z)dz -\int_0^y H(z\partial_z q)dz=yH(q)-\int_0^{y}H(z\partial_zq)\nonumber\\
&=yH(q)-H(yq)+\int_0^yH(q)=\int_0^y H(q).
\end{align}
Thus, \[|(I_6, yq_y)_{L^2_\phi}|=|(aF_a'\Lambda^{-1}q, yq_y)_{L^2_\phi}|=|a||(yF_a' \frac{1}{y}\Lambda^{-1}q, yq_y)_{L^2_\phi}|\leq |a| |yF_a'|_{L^2_\phi}\mathcal{E}(q)\leq |a|C\mathcal{E}(q).\]
Thus, we get:
 \begin{equation}\frac{1}{2}\frac{d}{ds}\|y\partial_yq\|^2_{L^2_\phi}\leq -\Big(\frac{1}{2}-C|a|\Big)|yq_y|_{L^2_\phi}^2+C|q|_{L^2_\phi}|yq_y|_{L^2_\phi}+C|a|\mathcal{E}(q)+C\mathcal{E}(q)^{3/2},\end{equation} with $C$ is a universal constant independent of $a$ (when $a$ is small enough). Now we first choose $a$ so small that we have: 
\[\frac{d}{ds}\|y\partial_yq\|^2_{L^2_\phi}\leq -\frac{1}{2}|yq_y|_{L^2_\phi}^2+C|q|_{L^2_\phi}^2+C|a|\mathcal{E}(q)+C\mathcal{E}(q)^{3/2},\]
\[\frac{d}{ds}\|q\|_{L^2_\phi}^2\leq -\frac{1}{2}|q|_{L^2_\phi}^2+|a|C\mathcal{E}(q)+C\mathcal{E}(q)^{\frac{3}{2}}.\] Next, we recall that \[\mathcal{E}(q)=\frac{1}{\delta}\|q\|_{L^2_\phi}^2+\|y\partial_yq\|^2_{L^2_\phi}.\] Thus we take $\delta$ so that $\frac{1}{\delta}>10C$ and use again that $a$ is small to see that \[\frac{d}{ds}\mathcal{E}(q)\leq -\frac{1}{4}\mathcal{E}(q)+C\mathcal{E}(q)^{3/2}.\]

Now  we prove the closure of the bootstrap.

\section{Proof of Proposition \ref{proposition:bootstrap}}
By using \eqref{Energyineq} and the bootstrap assumptions, one deduce that
\begin{align}
\frac{d}{ds}\Bigg(e^{\frac{s}{4}}\mathcal{E}(q)\Bigg)\leq CA_1^3e^{-\frac s8}.
\end{align}
Hence, by integrating the previous inequality between $s_0$ and $s$ we deduce that
\begin{align}\label{finalbound}
\mathcal{E}(s)\leq \mathcal{E}(s_0)e^{-\frac{(s-s_0)}{4}}+\frac{CA_1^3}{8}e^{-\frac s4}(e^{-\frac{s_0}{8}}-e^{-\frac{s}{8}}). 
\end{align}

Also, one has from \eqref{modbound} that
\begin{align}
&\frac{\lambda_s}{\lambda}+1=O\Big(e^{-\frac s8}\Big),\nonumber\\
&\frac{\mu_s}{\mu}=O\Big(e^{-\frac s8}\Big).
\end{align}
Hence, one can easily deduce that
\begin{align}\label{finalmodbound}
&\lambda(s)=Ce^{-s+O(e^{-\frac s8})},\nonumber\\
&\mu(s)=O\Big(e^{e^{-\frac s8}}\Big).
\end{align}

Let an initial datum satisfy \eqref{qini} at time $s_0$. Let $\tilde s$ be the supremum of times when the solution is trapped on $[s_0,\tilde s]$. Suppose that $\tilde s<+\infty$. Hence, from Definition \ref{definition:trap} and a continuity argument, one of the inequalities \eqref{parameterstrap} or \eqref{etrap} must be an equality at time $\tilde s$. This is  contradicting  \eqref{finalbound}, and \eqref{finalmodbound} for $K$ large enough which implies $\tilde s=+\infty$ and concludes the Proposition \ref{proposition:bootstrap}.

\section{Proof of Theorem \ref{theorem:main2}}\label{Proof of Theorem 2}
The proof of Theorem \ref{theorem:main2} is very similar to the proof of Theorem \ref{theorem:main} so we content ourselves with only giving a sketch. We discuss the two key elements which are different: a change of variables and the coercivity for the linearized operator. All the non-linear estimates are almost identical. 

We make the following change of variables in \eqref{eq:OSW}. First, since we are only looking at odd solutions, we consider the spatial domain to be $[0,\infty)$. For some $0<\alpha<1$ we define \[X=x^\alpha\] and set \[\omega(x,t)=\Omega(X,t)\] along with \[u(x,t)=x U(X,t).\]
Then the evolution equation in \eqref{eq:OSW} becomes:
\[\partial_t\Omega+a\alpha U X\partial_X \Omega=2U\Omega+2\alpha X\partial_XU \Omega.\]
Now let us study the relation between $U$ and $\Omega$. 
\[\partial_x(x U(X,t))=\partial_x u(x,t)=-H(\omega)(x,t)=-\frac{1}{\pi}PV\int_{-\infty}^\infty \frac{\omega(y,t)}{x-y}dy.\]
\[=-\frac{1}{\pi}PV\int_0^\infty \frac{y\omega(y,t)}{x^2-y^2}=-\frac{1}{\pi}PV\int_0^\infty \frac{y\Omega(Y,t)}{x^2-y^2}dy=-\frac{1}{\pi}\int_0^\infty \frac{Y^{1/\alpha}\Omega(Y,t)}{X^{2/\alpha}-Y^{2/\alpha}}d(Y^{1/\alpha})\]
\[=-\frac{1}{\pi\alpha}\int_0^\infty \frac{Y^{2/\alpha-1}}{X^{2/\alpha}-Y^{2/\alpha}}\Omega(Y,t)dY:=-\mathcal{H}_\alpha(\Omega)\] using the oddness of $\omega$. 
In particular, \[U+\alpha X\partial_X U=-\mathcal{H}_\alpha(\Omega).\]
Therefore, \[\partial_X(X^{1/\alpha}U)=-\frac{1}{\alpha}X^{1/\alpha-1}\mathcal{H}_\alpha(\Omega).\]
Now define\[\mathcal{L}_\alpha(f)=\frac{1}{\alpha}X^{-1/\alpha}\int_0^XY^{1/\alpha-1}f(Y)dY.\] Thus, \eqref{eq:OSW} becomes: 
\[\partial_t\Omega+a\alpha U X\partial_X \Omega=-2\Omega \mathcal{H}_\alpha(\Omega).\]
\begin{equation}\label{UtoOmega}U=-\mathcal{L}_\alpha \mathcal{H}_\alpha(\Omega).\end{equation}
\[\mathcal{H}_\alpha(\Omega)=\frac{1}{\pi\alpha}\int_0^\infty \frac{Y^{2/\alpha}}{X^{2/\alpha}-Y^{2/\alpha}}\frac{\Omega(Y,t)}{Y}dY.\]
Now, as shown in \cite{EJDG}, when $a=0$ for each $\alpha$, we have the following explicit self-similar profiles: \[\Omega(X,t)=\frac{1}{1-t}F^{(\alpha)}_{0}(\frac{X}{1-t}),\] where \[F^{(\alpha)}_{0}(Y)=-\frac{\sin(\frac{\alpha\pi}{2})Y}{1+2\cos(\frac{\alpha\pi}{2})Y+Y^2},\qquad \mathcal{H}_\alpha(F^{(\alpha)}_0)(Y)=\frac{1+\cos(\frac{\alpha\pi}{2})Y}{1+2\cos(\frac{\alpha\pi}{2})Y+Y^2}.\]
In particular, \[|\sin(\frac{\alpha\pi}{2})\tilde F_0-F_0^{(\alpha)}|_{H^3_Y}\leq C\alpha^2,\] where \[F_0(Y)=\frac{Y}{(1+Y)^2}\] as in the expression below \eqref{Faeq}. 
For the analysis we also need that \begin{equation}\label{Falphaexpansion} |\mathcal{H}_\alpha(F^{(\alpha)}_0)-\tilde HF_0|_{H^3_Y}\leq C\alpha,\end{equation} where\[\tilde H F_0(Y)=-\frac{1}{1+Y}.\] We also need that \[\alpha\|{\mathcal{H}_\alpha}\|_{H^1\rightarrow H^1}+ \|\mathcal{L}_\alpha\|_{H^1\rightarrow H^1}\leq C,\] where $C$ is independent of $\alpha$.
\subsection{Linearized Operator}
Following the proof of Theorem \ref{theorem:main}, we mainly need to establish coercivity properties of the main linearized operator. We thus content ourselves with establishing the analogue of Proposition \ref{proposition:spectralgap}.
We note that linearizing around $F_0^{\alpha}$  leads to \[\mathcal{M}_a^\alpha q=\mathcal{M}_0 q+P_{a}^\alpha q,\]
where 
\[\mathcal{M}_0^\alpha q=q+Y\partial_Y q+2\mathcal{H}_\alpha(F_0^{(\alpha)})q+2F_0^{(\alpha)}\mathcal{H}_\alpha(q).\]
and $P_a^\alpha(q)$ satisfies \[|(P_{a}^\alpha(q),q)_{\mathcal{E}}|\leq C(\alpha|a| +\alpha)\mathcal{E}(q)\] exactly as in the proof of Proposition \ref{proposition:spectralgap}. 

Now let us introduce the weight \[\phi_*(Y):=\frac{(Y+1)^4}{Y^4},\] which was used in \cite{E_Classical}. 
Then, recalling \eqref{Falphaexpansion}, we have \[(q+Y\partial_Yq+ 2\mathcal{H}_\alpha(F^{(\alpha)}_0)q, q \phi_*)\geq (\frac{1}{2}-C\alpha)|q\sqrt{\phi_*}|_{L^2}^2.\]
It remains to study $(\mathcal{H}_\alpha(q),q F_0^{(\alpha)} \phi_*)_{L^2}.$ 

\noindent {\bf Claim:} \[(\mathcal{H}_\alpha(q),q F_0^{(\alpha)} \phi_*)_{L^2}\geq C\alpha |q\sqrt{\phi_*}|_{L^2}^2.\]
Once the claim is established, the $L^2$ coercivity once $\alpha$ is small follows and the rest of the proof of Theorem \ref{theorem:main2} is similar to that of Theorem \ref{theorem:main}.
\[(\mathcal{H}_\alpha(q),q F_0^{(\alpha)} \phi_*)_{L^2}=\int_0^\infty\int_0^\infty\frac{Y^{2/\alpha}}{X^{2/\alpha}-Y^{2/\alpha}}\frac{(X+1)^2}{X^2}\frac{q(X)q(Y)}{XY}dXdY.\]
\[=\frac{1}{2}\int_0^\infty\int_0^\infty \frac{Y^{2/\alpha+2}(X+1)^2-X^{2/\alpha+2}{(Y+1)^2}}{X^{2/\alpha}-Y^{2/\alpha}}\frac{q(X)q(Y)}{X^3Y^3}dXdY.\] All we have done in the second equality is symmetrize the kernel. 
Now let us study the symmetrized kernel  \[K_\alpha(X,Y):=\frac{Y^{2/\alpha+2}(X+1)^2-X^{2/\alpha+2}{(Y+1)^2}}{X^{2/\alpha}-Y^{2/\alpha}}.\]
Observe that \[K_\alpha(X,Y)=\frac{\Big(\frac{Y}{X}\Big)^{1/\alpha}Y^2(X+1)^2-\Big(\frac{X}{Y}\Big)^{1/\alpha}X^2(Y+1)^2}{\Big(\frac{X}{Y}\Big)^{1/\alpha}-\Big(\frac{Y}{X}\Big)^{1/\alpha}}.\] 
Now, it is easy to see that $\lim_{\alpha\rightarrow 0^+} K(\alpha,X,Y)=-Y^2(X+1)^2{\bf 1}_{X<Y}-X^2(Y+1)^2{\bf 1}_{X\geq Y}$ on $\mathbb{R}^2\setminus\{X=Y\}$. Let us try to get some more quantitative information. 
By symmetry, we may restrict ourselves to the region where $X<Y.$ Observe that
\[K_{\alpha}(X,Y)+Y^2(X+1)^2=\Big(\frac{X}{Y}\Big)^{1/\alpha} \frac{Y^2(X+1)^2-X^2(Y+1)^2}{\Big(\frac{X}{Y}\Big)^{1/\alpha}-\Big(\frac{Y}{X}\Big)^{1/\alpha}}=\Big(\frac{X}{Y}\Big)^{1/\alpha} \frac{(2XY+X+Y)(Y-X)}{\Big(\frac{X}{Y}\Big)^{1/\alpha}-\Big(\frac{Y}{X}\Big)^{1/\alpha}}\]
\[=\frac{(2XY+X+Y)(Y-X)}{1-\Big(\frac{Y}{X}\Big)^{2/\alpha}}\]
\[=-X^2(2Y+1+\sigma)\frac{\sigma-1}{\sigma^{2/\alpha}-1},\] with $\sigma=\frac{Y}{X}$.
Defining $f(\sigma)=\sigma^{2/\alpha}-1$, let us note that $f,f',f''\geq 0$. Therefore, \[f(\sigma)\geq (\sigma-1)f'(1)=(\sigma-1)(\frac{2}{\alpha}-1)\] if $\sigma\geq 1$.
Thus, 
\[\frac{\sigma-1}{\sigma^{2/\alpha}-1}\leq \frac{\alpha}{2-\alpha}.\]
Consequently, 
\[|K_{\alpha}(X,Y)+Y^2(X+1)^2|\leq 2YX^2\alpha +2X^2\alpha+2XY\alpha \] if $\alpha\leq 1.$
Now let us note that \[\int_0^\infty \frac{|q|(X)}{X}\int_X^\infty \frac{|q|(Y)}{Y^3}dYdX\leq |q\sqrt{\phi_*}|_{L^2}\int_0^\infty \frac{|q|(X)}{|X|^2}dX=\int_0^\infty \frac{|q|(X)(X+1)^2}{X^2} \frac{1}{(X+1)^2}dX\leq 10|q\sqrt{\phi_*}|_{L^2}^2.\] The claim now follows once we show that \[\int_0^\infty\int_X^\infty Y^2(X+1)^2 \frac{q(X)q(Y)}{Y^3X^3}dXdY\leq 0,\] whenever $q(0)=q'(0)=\int_0^\infty \frac{q(Y)}{Y}dY=0.$
Indeed, 
\[\int_0^\infty\int_X^\infty Y^2(X+1)^2\frac{q(X)q(Y)}{Y^3X^3}dXdY=\int_0^\infty \frac{q(X)}{X}\frac{(X+1)^2}{X^2}\int_X^\infty \frac{q(Y)}{Y}dY\]\[=-\frac{1}{2}\int_0^\infty \frac{d}{dX}\Big(\int_X^\infty\frac{q(Y)}{Y}dY\Big)^2\frac{(X+1)^2}{X^2}dX\]\[=-\frac{1}{2}\int_0^\infty \frac{d}{dX}\Big(\int_X^\infty \frac{q(Y)}{Y}dY\Big)^2\Big(1+\frac{2}{X}+\frac{1}{X^2}\Big)dX\]\[=-\frac{1}{2}\int_0^\infty \Big(\int_X^\infty \frac{q(Y)}{Y}dY\Big)^2\Big(\frac{2}{X^2}+\frac{2}{X^3}\Big)\leq 0\]
\section{Appendix}

\subsection*{Weighted Identities}
\begin{lemma}\label{WeightedIdentity}
Let $\phi(y)=\frac{(1+y^2)^2}{y^4}.$ For all $f\in C_c^\infty(\mathbb{R})$ odd on $\mathbb{R}$ and satisfying $f'(0)=Hf(0)=0$, we have: \[\int |Hf(y)|^2 \phi(y)dy=\int |f(y)|^2\phi(y)dy.\] 
\end{lemma}
\begin{proof}
Note that $w(y)=1+\frac{2}{y^2}+\frac{1}{y^4}.$ Thus, it suffices to show that 
\[\int \frac{|Hf(y)|^2}{y^k}dy=\int \frac{|f(y)|^2}{y^k}dy\] under the condition that $f$ is odd and $f'(0)=Hf(0)=0$ for $k=0,2,4$. $k=0$ is just the isometry property of $H$. The case $k=2$ and $k=4$ are similar so we only do the more difficult case of $k=4$. 
Let us write $f=y^2 g$.  
Observe that, by assumption, we have $0=Hf(0)=\int y g$ and $\int g=0$ since $g$ is odd. Thus, \[H(y^2 g)=y^2 H(g)\]
In particular, we have: \[\int \frac{|Hf|^2}{y^4}dy=\int |H(g)|^2=\int |g|^2=\int \frac{|f|^2}{y^4}.\]
\end{proof}
\subsection*{Smoothing Procedure}

We now give a lemma which allows us to justify some of the computations in the coercivity and energy estimates. 
\begin{lemma}\label{Smoothing}
Let $q$ be such that $q$ is odd, $H(q)=q'(0)=0$ and $\mathcal{E}(q)<\infty$. Then, there exists a sequence $q_n\in \mathcal{S}(\mathbb{R})$ with 
\begin{itemize}
\item $q_n$ is odd and $q_n'(0)=Hq_n(0)=0$
\item $\mathcal{E}(q_n)\leq 2\mathcal{E}(q)$
\item $q_n\rightarrow q$ uniformly on $\mathbb{R}$. 
\end{itemize}
\end{lemma}
\begin{proof}
Take $q_n^1(y)=\frac{ny^2}{ny^2+1}\phi_n*q,$ where $\phi_n=\frac{n}{\sqrt{\pi}}\exp(-n^2x^2).$ Clearly, $q_n^1$ is odd and ${q_n^1}'(0)=0$. Moreover, $q_n^1\in\mathcal{S}(\mathbb{R})$ and $q_n^1\rightarrow q$ uniformly on $\mathbb{R}$. It may be, however, that $H(q_n^1)(0)\not =0$. Now let's define the function \[\psi(y)=y^3\exp(-y^2).\] Clearly, $\psi\in\mathcal{S}$ and $\mathcal{E}(\psi)<\infty$. Moreover, \[H(\psi)(0)\not=0.\]
Thus we define: \[q_n(y)=q_n^1(y)-\frac{H(q_n^1)(0)}{H(\psi)(0)}\psi(y).\] Clearly, $q_n$ is odd and $q_n'(0)=Hq_n(0)=0$. Now let's compute $\mathcal{E}(q_n)$.
First, $H(q_n^1)(0)\rightarrow 0$ as $n\rightarrow\infty$ by the dominated convergence theorem and $\mathcal{E}(\psi)<\infty$. Thus, the second term in the definition of $q_n$ converges uniformly to $0$ in $\mathbb{R}$ in the energy norm. It is also easy to see that if $n$ is large enough, $\mathcal{E}(q_n^1)\leq 2\mathcal{E}(q).$
\end{proof}

\bibliographystyle{plain}
\bibliography{3dEuler_away_final2}

\end{document}